\documentclass{amsart}

\usepackage[cmtip,all]{xy}
\usepackage{amsmath, amsthm, latexsym, amssymb, amsfonts}
\usepackage{alltt}
\usepackage[pdftex]{graphicx,color} 
\usepackage{algorithm}
\usepackage{algpseudocode}
\usepackage{longtable}
\usepackage[table]{xcolor}
\usepackage{tikz}
\usepackage{tikzscale}
\usepackage{pgfplots}
\pgfplotsset{compat=1.16} 
\usetikzlibrary{shapes}
\usepackage{mathtools}

\newcommand{\dis}{\displaystyle}

\newcommand{\Spec}{\operatorname{Spec}}
\newcommand{\Hom}{\operatorname{Hom}}

\newcommand{\h}{\operatorname{height}}

\newcommand{\reg}{\operatorname{reg}}
\newcommand{\Cl}{\operatorname{Cl}}

\newcommand{\sig}{\sigma}
\newcommand{\Sig}{\Sigma}

\newcommand{\lcm}{\operatorname{lcm}}

\newcommand{\Z}{\mathbb Z}
\newcommand{\N}{\mathbb N}
\newcommand{\R}{\mathbb R}
\newcommand{\F}{\mathbb F}
\newcommand{\Pp}{\mathbb P}
\newcommand{\K}{{\mathbb K}}

\newcommand{\cC}{\cl C}
\newcommand{\cA}{\cl A}
\newcommand{\cl}[1]{\mathcal{#1}}

\newcommand{\la}{\langle}
\newcommand{\ra}{\rangle}
\def\aa{{\bf \alpha}}
\def\bb{\beta}
\def\a{{\bf a}}

\def\q{{\bf q}}

\def\uu{{\bf u}}
\def\vv{{\bf v}}
\def\x{{\bf x}}
\def\m{{\bf m}}

\def\b{{\bf b}}
\def\e{{\bf e}}
\def\ev{{\text{ev}}}

\newtheorem{tm}{Theorem}[section]

\newtheorem{corollary}[tm]{Corollary}

\newtheorem{proposition}[tm]{Proposition}
\newtheorem{definition}[tm]{Definition}

\newtheorem{example}[tm]{Example}
\newtheorem{remark}[tm]{Remark}

\numberwithin{equation}{section}

\begin{document}

\title{Codes on Subgroups of Weighted Projective Tori}


\author{Mesut \c{S}ah\.{i}n}
\address{ Department of Mathematics,
  Hacettepe  University,
  Ankara, TURKEY}
\curraddr{}
\email{mesut.sahin@hacettepe.edu.tr}
\thanks{The authors are supported by T\"{U}B\.{I}TAK Project No:119F177}

\author{O\u{g}uz Yayla}
\address{Institute of Applied Mathematics, Middle East Technical University, Ankara, Turkey}
\email{oguz@metu.edu.tr}


\date{}

\dedicatory{}

\begin{abstract}
We obtain certain algebraic invariants relevant in studying codes on subgroups of weighted projective tori inside an $n$-dimensional weighted projective space. As application, we compute all the main parameters of generalized toric codes on these subgroups of tori lying inside a weighted projective plane of the form $\Pp(1,1,a)$. 
\end{abstract}

\maketitle

\section{Introduction} 

Let $\Pp(w_0,\dots,w_n)$ be the weighted projective space over an algebraic closure $\overline{\F}_q$ of a finite field $\F_q$, defined by some positive integers $w_0,\dots,w_n$. Without loosing generality, we assume that $n$ of these numbers have no common divisor. It is well known that the $\overline{\F}_q$-rational points of the weighted projective space $\Pp(w_0,\dots,w_n)$ can be represented by the Geometric Invariant Theory quotient $(\overline{\F}_q^{n+1}\setminus \{0\})/G$, where the group $G=\{(\lambda^{w_0},\dots,\lambda^{w_n}) :  \lambda\in \overline{\F}_q^* \}$. Therefore, a point is an orbit of the form $[p_0:\ldots:p_n]=\{(\lambda^{w_0}p_0,\dots,\lambda^{w_n}p_n)\: : \: \lambda\in \overline{\F}_q^* \}$ known as its homogeneous coordinates 
 as in the classical projective case.  Every $\F_q$-rational point has a representative from the set $\F_q^{n+1}$ in this correspondence.

For a thorough introduction to and a fairly good account on general properties of these spaces, see \cite{GhorpadeEtAllWPS,BRWPS,DolgachevWPS,RossiTerraciniWPS12}. It is known that $X=\Pp(w_0,\dots,w_n)$ is smooth if and only if it is the usual projective space $\Pp^{n}$, i.e., $w_0=\dots=w_n=1$.

The  coordinate ring $S=\F_q[x_0,\dots,x_n]$ over the field $\F_q$ of a weighted projective space $\Pp(w_0,\dots,w_n)$ is graded naturally by the numerical semigroup $\N \bb$ generated by $\deg (x_i)=w_i$, for $i=0,\dots,n$, where $\N$ denotes the set of natural numbers with $0$. Thus, we have the following decomposition:
$$S=\displaystyle \bigoplus_{\aa \in \N \bb } S_{\aa}, \text {where $S_{\aa}$ is the vector space spanned by the monomials of degree $\aa$}.$$

For any $\aa \in \N \bb$ and any subset $Y=\{P_1,\dots,P_N\}$ of $\F_q$-rational points  of $\mathbb{P}(w_0,\ldots,w_n)$, we have the following {\it evaluation map}:
\begin{equation} \label{e:evaluationMap}
\dis \ev_{Y}:S_\aa\to \F_q^N,\quad F \mapsto \left(F(P_1),\dots,F(P_N)\right).
\end{equation}
The image $\cl{C}_{\aa,Y}=\text{ev}_{Y}(S_\aa)$ is a linear code. The three basic parameters of $\cl C_{\aa,Y}$ are block-length which is $N$, the dimension which is $K=\dim_{\F_q}(\cl C_{\aa,Y})$, and the minimum distance $\delta=\delta(\cl{C}_{\aa,Y})$ which is the minimum of the number of nonzero components of vectors in $\cl C_{\aa,Y}\setminus \{0\}$. When $Y$ is the full set of $\F_q$-rational points of $\Pp(w_0,\dots,w_n)$, the code is known as the {\it weighted Reed-Muller code}. These codes are special cases of what is called generalized toric codes, see Section \ref{sec:prelim} for details.

Toric codes are introduced by Hansen in \cite{HansenAAECC} for the set $Y=T_X(\F_q)$ of $\F_q$-rational points of the dense torus $T_X$ of a toric variety $X=X_{\Sigma}$ and examined further in e.g. \cite{JoynerAAECC,DR2009,OrderBound,ToricCIcodesIS13,LittleFFA2013,7championCodes} producing some codes having the best known parameters. The vanishing ideal $I(Y)$ of $Y$ which is generated by homogeneous polynomials vanishing on $Y$ is a key in studying the parameters  of $\cl C_{\aa,Y}$. This is because, the kernel of $\ev_{Y}$ is nothing but the subspace $I_{\aa}(Y):=I(Y)\cap S_{\aa}$, and hence the code $\cl C_{\aa,Y}$ is isomorphic to the vector space $S_{\aa}/I_{\aa}(Y)$. Therefore, the dimension $K=\dim_{\F_q}(\cl C_{\aa,Y})$ is the value $H_Y(\aa)=\dim_{\F_q}S_{\aa}-\dim_{\F_q}I_{\aa}(Y)$ of the multigraded Hilbert function $H_Y$ of $Y$, see \cite{SaSop16JA}. Most recently, Nardi developed combinatorial methods for studying codes on the full set $Y=X(\F_q)$ of $\F_q$-rational points of a toric variety, see \cite{NardiHirzebruch,NardiProjToric}.

In literature, there are a few papers computing the main parameters of codes on weighted projective spaces. The main parameters of some weighted Reed-Muller codes are given explicitly for  the set $Y=X(\F_q)$ of $\F_q$-rational points of the weighted projective planes $X=\Pp(1,w_1,w_2)$ when $\aa$ is a multiple of the $\lcm(w_1,w_2)$. The main parameters have the most beautiful formulas in the special case of the plane $X=\Pp(1,1,a)$, see  \cite{GhorpadeEtAllWPS}.

 If $Y = T_X(\F_q)=\{[1:t_1:\ldots :t_n] \: | \: t_i \in \F_q^*, \text{ for all } i\in [n]:=\{1,\ldots,n\} \}$ is the set of $\F_q$-rational points of the  torus $T_X$ in $X=\Pp^{n}$ and $\aa \geq 1$, then the main parameters are given in \cite{SVPV2011MinDisProjTori}.
On the other hand, \cite{SM2013MinDisProjDegTori} studied the degenerate tori $$Y_Q=\{[1:t_1^{a_1}:\ldots :t_n^{a_n}] \: | \: t_i \in \F_q^*, \text{ for all } i\in [n]:=\{1,\ldots,n\} \}$$ 
 lying in the classical projective space $X=\Pp^n$, generalizing \cite{SVPV2011MinDisProjTori}. This is because, $Y_Q$ becomes the  set of $\F_q$-rational points of the projective torus in $\Pp^{n}$, once $a_i=1$, for all $i\in [n]$. The results in \cite{SM2013MinDisProjDegTori} show that $I(Y_Q)$ is a complete intersection of the binomials $x_i^{s_i}-x_0^{s_i}$, for $i\in [n]$, its degree is $|Y_Q|=s_1\cdots s_n$ and $a$-invariant is $a_Y=s_1+\cdots+s_n-n-1$, where $s_i=(q-1)/\gcd(q-1,a_i)$ for all $i\in [n]$. Some nice formulas are given for the other parameters as well.

The present paper considers the analogue of the same parametrization $Y_Q$ but in the weighted projective space $X=\Pp(1,w_1,\dots,w_n)$ with $a_i=w_i$ for all $i$. When $w_i=1$, for all $i$, our $Y_Q$ becomes the $\F_q$-rational points of the projective torus studied in \cite{SVPV2011MinDisProjTori}, as well. In the next section, we review basic terminology and theory needed in the sequel. We prove that $I(Y_Q)$ is a complete intersection ideal in Proposition \ref{prop:ideal_Y_Q}. We give a formula for the Hilbert function $H_{Y_Q}$ and compute the $a$-invariant of $Y_Q$ in Proposition \ref{P:HilbertF}. Theorem \ref{t:dimension} gives formulas for the length and dimension of the code $\cl C_{\aa,Y_Q}$.  The final section displays more explicit formulas for the dimension and minimum distance of the codes coming from the weighted projective plane $\Pp(1,1,a)$, see Theorem \ref{t:codesOnDegenToriP(1,1,a)}. 

\section{Preliminaries}
\label{sec:prelim}

Let $\Sigma\subseteq \R^n$ be a complete simplicial fan with rays generated by the lattice vectors $\vv_1,\dots,\vv_r$. Each cone $\sigma \in \Sigma$, defines an affine toric variety $U_{\sigma}=\Spec (\K[ \check{\sig}\cap \Z^n])$ over an algebraically closed field $\K$. Gluing these affine pieces, we obtain the toric variety $X_{\Sig}$ as an abstract variety over $\K$. There is a nice correspondence between polytopes in real $n$-space and projective toric varieties. Namely, every lattice polytope $\cl P$ gives rise to a so called normal fan $\Sigma_{\cl P}$ whose rays are spanned by the inner normal vectors of $\cl P$. Assuming $X_{\Sigma}$ has a free class group, the ray generator yields the following short exact sequence: 
$$\dis \xymatrix{ \mathfrak{P}: 0  \ar[r] & \Z^n \ar[r]^{\phi} & \Z^r \ar[r]^{{\bb}} & \Z^d \ar[r]& 0},$$  
where $\phi$ is the matrix  $[\vv_1\cdots \vv_r]^T$ and $d=r-n$ is the rank of the class group $\Cl X_{\Sig}\cong \Z^d$. There is an important lattice $L_{\bb}$ in $\Z^r$ that is isomorphic to $\Z^n$ via $\phi$, and is spanned by the columns $\uu_1,\dots, \uu_n$ of $\phi$. 

Applying $\Hom(-,\K^*)$ functor to $\mathfrak{P}$ gives the following dual short exact sequence:
$$\dis \xymatrix{ \mathfrak{P}^*: 1  \ar[r] & G \ar[r]^{i} & (\K^*)^r \ar[r]^{\pi} & (\K^*)^n \ar[r]& 1}, $$
where $\pi(P)=(\x^{\uu_1}(P), \dots , \x^{\uu_n}(P))$ and $\x^{\a}(P)=p_1^{a_{1}}\cdots p_r^{a_{r}}$ for $P=(p_1,\dots,p_r)\in (\K^*)^r$ and $\a=(a_1,\dots,a_r)\in \Z^r$. 

As proved by Cox in \cite{CoxRingToric}, the set $X(\K)$ of $\K$-rational points of the toric variety $X:=X_{\Sigma}$ is identified with the geometric quotient $[\K^r\setminus V(B)] /G$, where $B$ is the monomial ideal in $\K[x_1,\dots,x_r]$ generated by the monomials $\dis \x^{\hat{\sig}}=\Pi_{\rho_i \notin \sig}^{} \: x_i$ corresponding to cones $\sig \in \Sig$.
Hence, points of $X(\K)$ are orbits $[P]:=G\cdot P$, for $P\in \K^r\setminus V(B)$. When $\K=\overline{\F}_q$ is an algebraic closure of a finite field $\F_q$, the $\F_q$-rational points $[P]$ are represented by points $P$ from the set $\F_q^r\setminus V(B)$.

 The coordinate ring $S=\F_q[x_1,\dots,x_r]$ of $X$ is graded via the columns of the matrix $\bb$, i.e. $\deg_{\bb}(x_j)=\bb_j$, for $j=1,\dots,r$. There is a nice correspondence between subgroups of the torus $T_X(\F_q)\cong (\F_q^*)^r/G$ and $\bb$-graded \textit{lattice ideals} in $S$, defined by: 
$$I_L=\la \x^{\m^+}- \x^{\m^-} \: | \: \m=\m^{+}-\m^{-}\in L \ra,$$
where $L$ is a sublattice of $L_{\bb}$, see \cite{sahinFFA22}.
 In the case of the weighted projective space $\Pp(w_0,\dots,w_n)$, we have the row matrix $\bb=[w_0 \cdots w_n]$.
\begin{figure}[htb!]
\centering
\begin{minipage}{0.5\textwidth}
  \centering
\begin{tikzpicture}[scale=0.95]

\draw[dashed,lightgray][line width=0mm] (-1,-1.5)  -- (-1,3.5);
\draw[dashed,lightgray][line width=0mm] (0,-1.5)  -- (0,3.5);
\draw[dashed,lightgray][line width=0mm] (1,-1.5)  -- (1,3.5);
\draw[dashed,lightgray][line width=0mm] (2,-1.5)  -- (2,3.5);
\draw[dashed,lightgray][line width=0mm] (3,-1.5)  -- (3,3.5);
\draw[dashed,lightgray][line width=0mm] (4,-1.5)  -- (4,3.5);
\draw[dashed,lightgray][line width=0mm] (-1,2)  -- (4,2);
\draw[dashed,lightgray][line width=0mm] (-1,3)  -- (4,3);
\draw[dashed,lightgray][line width=0mm] (-1,1)  -- (4,1);
\draw[dashed,lightgray][line width=0mm] (-1,0)  -- (4,0);
\draw[dashed,lightgray][line width=0mm] (-1,-1)  -- (4,-1);
\node[inner sep=0,anchor=west,{scale=0.7}] (note1) at (-0.9,2.2) {$(0,2)$};
\node[inner sep=0,anchor=west,{scale=0.7}] (note1) at (-0.9,0.2) {$(0,0)$};
\node[inner sep=0,anchor=west,{scale=0.7}] (note1) at (3.1,0.2) {$(3,0)$};
\draw (0,0) node[circle,fill,blue,inner sep=1pt] {};
\draw (0,2) node[circle,fill,blue,inner sep=1pt] {};
\draw (3,0) node[circle,fill,blue,inner sep=1pt] {};
\draw[thick,->][line width=0.4mm] (0,4/3) -- (0.5,4/3);
\draw[thick,->][line width=0.4mm] (1,0) -- (1,0.5);
\draw[thick,->][line width=0.4mm] (1,4/3) -- (1-0.4,4/3-0.6);

\draw (0,0)-- (3,0)--(0,2)--cycle;
\end{tikzpicture} 
  \caption{The polygon $\cl P$}
  \label{fig:polygon}
\end{minipage}%
\begin{minipage}{0.5\textwidth}
  \centering
\begin{tikzpicture}[scale=0.95]
\draw[dashed,lightgray][line width=0mm] (0,-3.5)  -- (0,1.5);
\draw[dashed,lightgray][line width=0mm] (-3,0)  -- (2,0);
\draw[dashed,lightgray][line width=0mm] (-3,-3.5)  -- (-3,1.5);
\draw[dashed,lightgray][line width=0mm] (-2,-3.5)  -- (-2,1.5);
\draw[dashed,lightgray][line width=0mm] (-1,-3.5)  -- (-1,1.5);
\draw[dashed,lightgray][line width=0mm] (1,-3.5)  -- (1,1.5);
\draw[dashed,lightgray][line width=0mm] (2,-3.5)  -- (2,1.5);
\draw[dashed,lightgray][line width=0mm] (-3,1)  -- (2,1);
\draw[dashed,lightgray][line width=0mm] (-3,-1)  -- (2,-1);
\draw[dashed,lightgray][line width=0mm] (-3,-2)  -- (2,-2);
\draw[dashed,lightgray][line width=0mm] (-3,-3)  -- (2,-3);
\draw[thick,->][line width=0.4mm] (0,0) -- (-2,-3);
\draw[thick,->][line width=0.4mm] (0,0) -- (0,1);
\draw[thick,->][line width=0.4mm] (0,0) -- (1,0);
\draw[red][line width=0.4mm] (-2,-3)  -- (-2.4,-3.6);
\draw (-2,-3) node[circle,fill,blue,inner sep=1pt] {};
\draw[red][line width=0.4mm] (0,1)  -- (0,1.5);
\draw (0,1) node[circle,fill,blue,inner sep=1pt] {};
\draw[red][line width=0.4mm] (1,0)  -- (2,0);
\draw (1,0) node[circle,fill,blue,inner sep=1pt] {};
\node[inner sep=0,anchor=west,{scale=0.7}] (note1) at (-0.9,1.2) {$(0,1)$};
\node[inner sep=0,anchor=west,{scale=0.7}] (note1) at (-0.9,0.2) {$(0,0)$};
\node[inner sep=0,anchor=west,{scale=0.7}] (note1) at (1.1,0.2) {$(1,0)$};
\node[inner sep=0,anchor=west,{scale=0.7}] (note1) at (-1.7,-3.2) {$(-2,-3)$};
\end{tikzpicture}
  \caption{The fan $\Sigma_{\cl P}$}
  \label{fig:normalfan}
\end{minipage}
\end{figure}

\begin{example} \label{ex:P(123)} Let $X=\Pp(1,2,3)$ be the weighted projective space over $\overline{\F}_3$, which corresponds to the normal fan $\Sigma_{\cl P}$ depicted in Figure \ref{fig:normalfan} of the polygon $\cl P$ depicted in Figure \ref{fig:polygon}. Then, the first sequence above becomes:
$$\dis \xymatrix{ \mathfrak{P}: 0  \ar[r] & \Z^2 \ar[r]^{\phi} & \Z^3 \ar[r]^{\beta}& \Z \ar[r]& 0},$$ 
where $$\phi=\begin{bmatrix}
 -2 & 1 & 0 \\
-3 & 0 &  1
\end{bmatrix}^T   \quad  \mbox{ and} \quad \beta=\begin{bmatrix}
1 & 2 & 3
 \end{bmatrix}.$$ 
The coordinate ring $S= \F_3[x,y,z]$ is multigraded via
 $$\deg_{\bb}(x)=1,\ \deg_{\bb}(y)=2 \ \mbox{ and }  \deg_{\bb}(z)=3.$$
Since $B=\la x,y,z\ra$, we remove the set $V(B)=V(x,y,z)=\{0\}$ and therefore obtain the quotient representation $X(\F_3)=(\F_3^3 \setminus 0)/G$, where  
 $$G=\{(x,y,z)\in (\overline{\F}_3^*)^3 \:|\: x^{-2}y=x^{-3}z=1 \}=\{(\lambda,\lambda^{2},\lambda^3) \:|\: \lambda \in \overline{\F}_3^*\}$$
is the zero locus in $(\overline{\F}_3^*)^3$ of the toric ideal:
 \begin{align*}
I_{L_\bb}:=\langle \x^{\uu}-\x^{\vv} \: : \: \uu,\vv \in \N^r \:\mbox{and} \: \bb\uu=\bb\vv  \rangle =\langle x^{2}-y, x^{3}-z\rangle.   
\end{align*}
One needs to be careful about the field over which the group $G$ is considered. Even though we use representative from the affine space  $\F_3^3$ recall that the equivalence of points in an orbit is determined via the subgroup $G$ of $(\overline{\F}_3^*)^3$. For instance, the points $[0:0:1]$ and $[0:0:2]$ are the same as $\F_3$-rational points, since there is $\lambda \in \overline{\F}_3^*$ such that $\lambda^2=2$ and thus we have $(0,0,2)=(\lambda,\lambda,\lambda^2)\cdot (0,0,1)$. But, these two points would be different if we considered equivalence with respect to the existence of $\lambda \in \F_3^*$ such that $\lambda^2=2$, since $\lambda^2=1$ for all $\lambda \in \F_3^*=\{1,2\}$.   
 \end{example}

Let us recall basics of linear codes. Our alphabet is the finite field $\F_q$ with $q$ elements. A \textit{linear code} is a subspace ${\cC}\subset \F_q^N$ whose elements are referred to as the  \textit{codewords}.

\begin{definition}

\textit{ The parameters of a linear code ${\cC}\subset \F_q^N$ are as follows:} 
\begin{itemize}
\item $N$ is the  \textit{length} of $\cC$, \\
\item $K=\dim_{\F_q} \cC$ is the \textit{dimension}  of $\cC$ as a subspace (a measure of \textit{efficiency}),\\
\item  $\delta$ is the \textit{minimum distance} of $\cC$ (a measure of \textit{reliability}), which is the minimum of all Hamming distances between different codewords in $\cC$, where  the Hamming distance between two codewords $c_1$ and $c_2$ is 
\[dist(c_1,c_2):=\# \text{of non-zero entries in } c_1-c_2.\] So,
$$\delta(\cC)=\min_{c\in{\cC}\setminus\{0\}}(\# \text{of non-zero entries in } c).$$ 
\end{itemize}
\end{definition}

As in Equation \eqref{e:evaluationMap}, we get the so called \textit{generalized toric codes} by evaluating homogeneous polynomials $F\in S_{\aa}$ of degree $\aa$ at some subset $Y$ of $\F_q$-rational points in a toric variety $X$.

\begin{definition} Let $Y\subseteq X$ be a subset of a toric variety $X$. Its vanishing ideal $I(Y)$ is the (homogeneous) ideal in $S$ generated by homogeneous polynomials vanishing on $Y$. The \it{multigraded Hilbert function} of $Y$ is 
\[H_{Y}(\aa):=\dim_{\K}{S}_{\aa}-\dim_{\K} I_{\aa}(Y).\]
\end{definition}

Since, the kernel of the evaluation map in Equation (\ref{e:evaluationMap}) consists of the homogeneous polynomials of degree $\aa$ whose image is the point $(0,\dots,0)\in \F_q^N$, it follows that the dimension of the code $\cC_{\aa,Y}$ equals the value $H_Y(\aa)$ of the Hilbert function of $Y$. When $Y$ lies in the torus $T_X$, the variables $x_i$ are all non zero-divisors in the quotient ring $S/I(Y)$, and thus the Hilbert function does not decrease as we state in the following result. Below we use the partial ordering $\preceq $, where $\aa \preceq  \aa'$ if
$\aa' - \aa \in \N \bb$. Notice that this is the usual ordering in $\N$ for $X:=\Pp(1,w_1\dots,w_{n})$ as $\N \bb=\N$ in this case.
\begin{proposition}{\cite[Corollary 3.18]{SaSop16JA}}
Let $Y\subset T_X$. The dimension $H_Y(\aa)$ of
$\cC_{\aa,Y}$ is non-decreasing in the sense that $H_Y(\aa)\leq H_Y(\aa')$ for all $\aa \preceq \aa'$.
\end{proposition}

On the other hand, the minimum distance behaves the opposite way as the following points out:

\begin{proposition}{\cite[Proposition 2.22]{Sahin20INDAM}}
Let $Y\subset T_X$. The minimum distance of
$\cC_{\aa,Y}$ is non-increasing in the sense that $\delta(\cC_{\aa,Y}) \geq \delta(\cC_{\aa',Y})$ for all $\aa \preceq \aa'$.
\end{proposition}

These two results are not that surprising as we have the following well known relation between these two parameters given by the Singleton's bound:

$$\delta(\cC_{\aa,Y}) + K(\cC_{\aa,Y}) \leq N(\cC_{\aa,Y})+1.$$

There is an algebro-geometric invariant of the zero-dimensional subvariety $Y\subset X(\F_q)$ used to eliminate trivial codes which we introduce now.

\begin{definition} 
 The \textit{ multigraded regularity} of $Y$, denoted $\reg(Y)$, is the set of $\aa\in  \N \bb$ for which $H_Y(\aa)=|Y|$, the length of $\cC_{\aa,Y}$.
\end{definition} 

\begin{proposition} \label{trivialCodes}
If $\aa \in \reg(Y)$ then $\delta(\cC_{\aa,Y})=1$. 
\end{proposition} 
\begin{proof}
Let $\aa \in \reg(Y)$. Then, the dimension of the code is nothing but the length. So, the claim follows from the Singleton bound, as we always have $\delta(\cC_{\aa,Y}) \geq 1.$ 
\end{proof}

 The multigraded regularity set is determined by a number also known as the $a$-invariant in the case of a weighted projective space. In order to state the precise result, we first recall some relevant concepts.

When $I$ is a weighted graded ideal, the quotient ring $S/I$ inherits this grading as well and has a decomposition $\dis S/I=\bigoplus_{\aa \in \cA} ({S/I})_{\aa}$, where $({S/I})_{\aa}=S_{\aa}/I_{\aa}$ is a finite dimensional vector space spanned by monomials of degree $\aa$ in the numerical semigroup $ \N \bb=\N\{w_0,\dots,w_n\}$, which do not belong to $I$. This gives rise to the weighted Hilbert function and series defined respectively by 
$$H_{S/I}(\aa) :=\dim_{\K}(S/I)_{\aa}=\dim_{\K}{S}_{\aa}-\dim_{\K} I_{\aa}$$
$$\mbox{and}\quad HS_{S/I}(t) :=\sum_{\aa \in  \N \bb} H_{S/I}(\aa) t^{\aa}.$$
Furthermore, the weighted Hilbert series has a rational function representation, that is, we have
\begin{equation}\label{e:HilbSeries}
HS_{S/I}(t)=\frac{p_{S/I}(t)}{(1-t^{w_0})\cdots (1-t^{w_n})},
\end{equation}
for a unique polynomial $p_{S/I}(t)$ with integer coefficients, see \cite[Chapter 8]{CombComAlgBook}.
\begin{proposition} \cite[Proposition 3.12]{SaSop16JA} \label{p:a-invariant}
Let $Y\subset T_X$ for $X=\Pp(w_0,\dots,w_{n})$ with $w_0=1$. Then, the integer $a_Y=\deg(p_{S/I(Y)}(t))-w_0-\cdots-w_n$ satisfies $\reg(Y)=1+a_Y+\N$.
\end{proposition} 

A nice formula for the $a$-invariant is given for the $\F_q$-rational points of the torus $T_X$ when $X$ is a weighted projective space.

\begin{proposition} \cite[ Corollary 3.9]{CodesWPS} \label{p:aInvariantWPTori}
If $Y=T_X(\F_q)$ for $X=\Pp(w_0,\dots,w_{n})$ and $g(\N \bb)$ is the Frobenius number of the numerical semigroup $\N \bb=\N\{w_0,\dots,w_n\}$, then 
$$a_Y=(q-2)[w_0+\cdots+w_n+g(\N \bb)]+g(\N \bb).$$ 
 
\end{proposition}  

There are subgroups of the torus $T_X$ referred to as degenerate tori which we briefly discuss now. 

\begin{definition}
The following subgroup
$Y_A =\{[t_1^{a_1}:\ldots:t_r^{a_r}] \: : \: t_i \in \F_q^* \}$ of the torus $T_X$ is called a \textit{degenerate torus}, lying inside a toric variety $X_{\Sigma}$,  for any positive integers $a_1,\dots,a_r$, where $r$ is the number of rays in the fan $\Sigma$. 
\end{definition}

If $\F_q^*=\la \eta \ra$, every $t_i \in \F_q^*$ is of the form $t_i=\eta^{k_i}$, for some $0 \leq k_i \leq q-2$. Let $d_i=| \eta^{a_i} |$ and $D=diag(d_1,\dots,d_r)$.

\begin{proposition} \cite[Corollary 3.13 (ii)]{HLRVLZ2013} \label{p:aInvariantPTori}
If  $Y=Y_A$ is a complete intersection in $X=\Pp^{r-1} $ and $g:=\gcd(d_1,\dots,d_r)$ so that $d'_1=d_1/g,\dots, d'_r=d_r/g$ generate a numerical semigroup $\N D'$ with the Frobenius number $g(\N D')$, then
$$1+a_Y=g\cdot g(\N D')+d_1+\cdots+d_r-(r-1).$$ 
 
\end{proposition}

Notice that when $a_i=1$ and $w_j=1$, for all $i$ and $j$, we have $d_i=q-1$, and so $d_i'=1$. The greatest integer not belonging to the numerical semigroup $\N \bb=\N D'=\N$ is $g(\N \bb)=g(\N D')=-1$ so both formulas in Proposition \ref{p:aInvariantWPTori} and Proposition \ref{p:aInvariantPTori} yield $a_Y=n(q-2)-1$, for the torus $Y=T_X(\F_q)$ in the projective space $X=\Pp^n$.

\begin{definition} A binomial is  a polynomial of the form $\x^{\a}-\x^{\b} $, and $J$ is called a \textit{binomial ideal} if it is generated by binomials. $J$ is called a complete intersection if it is generated by $\h(J)$ many  binomials.
\end{definition} 

\begin{definition} For a lattice $L\subset \Z^r$,  the \textit{lattice ideal} $I_L$ is the binomial ideal generated by binomials $\x^{\a}- \x^{\b} $ for all $\a-\b \in L $.  That is, $$I_L=\langle \x^{\a}- \x^{\b}  \: | \: \a-\b \in L \rangle \subset S.$$ 

\end{definition}

\begin{tm}\cite[Theorem 4.5]{sahin18} \label{t:IdealDegTori} If $Y=Y_A$ then $I(Y)=I_L$ for  $L=D(L_{\beta D})$. 

\end{tm}

If $a_i=1$, for all $i$, then $Y_A=T_X(\F_q)$ and $d_i=q-1$, for all $i$, so that the matrix $D$ is  just $q-1$ times the identity matrix yielding the following:

\begin{corollary} \cite[Corollary 4.14 (ii)]{sahin18} If $Y=T_X(\F_q)$ then $I(Y)=I_{L}$ for $L=(q-1)L_{\beta}$.

\end{corollary} 

\begin{proposition}\cite[Proposition 4.12]{sahin18} \label{p:IdealDegTori} A generating system of binomials for $I(Y_{A})$ is obtained from that of $I_{L_{\beta D}}$ by replacing $x_i$ with $x_i^{d_i}$. $I(Y_{A})$ is a complete intersection if and only if so is the toric ideal $I_{L_{\beta D}}$. In this case, a minimal generating system is obtained from a minimal generating system of $I_{L_{\beta D}}$ this way.
\end{proposition}

\section{Degenerate Tori on Weighted Projective Spaces}

In this section, we explore properties of some degenerate tori on a weighted projective space. To start with, we prove that they are complete intersections of special type of binomial hypersurfaces.

We focus on a weighted projective space $X=\Pp(w_0,\dots,w_n)$ and use the notation $S=\F_q[x_0,\dots,x_n]$ for the Cox ring of $X$. Set

$$\Tilde{w}_i := \frac{w_i}{\gcd(q-1,w_i)} \mbox{ and } d_i := \frac{q-1}{\gcd(q-1,w_i)} \mbox{ for } i=0,1,\ldots,n. $$ 

The following concept is very helpful in determining when a lattice ideal is a complete intersection.

\begin{definition} If each column of a matrix has both a positive and a negative entry we say that  the matrix is mixed. Moreover, if the matrix does
not have a square mixed submatrix, then it is called {\it dominating}. 
\end{definition} 

\begin{tm}\label{t:mixeddominating}\cite[Theorem 3.9]{MT2005} Let $L\subseteq\Z^r$ be a lattice with the property that $L\cap \N^r={0}$. Then, $I_L$ is  a  complete intersection  if and only if $ L$ has a basis ${\m_1,\dots,\m_k}$ such that the matrix $[\m_1\cdots \m_k]$ is mixed dominating.  If $I_L$ is a  complete intersection, then we have 
$$ I_L=\la \textbf{x}^{\textbf{m}_1^+}-\textbf{x}^{\textbf{m}_1^-},\dots, \textbf{x}^{\textbf{m}_k^+}-\textbf{x}^{\textbf{m}_k^-} \ra. $$
\end{tm}

\begin{proposition}\label{prop:ideal_Y_Q} Let $Q=diag(w_0,\ldots,w_n)$ and $Y_Q=\{[{t_0}^{w_{0}}:\ldots:{t_n}^{w_{n}}]|t_i\in \F_q^{*}\}$ be the corresponding subgroup of $T_X$ for $X=\Pp(w_0,\ldots,w_n)$.
 If $w_0 \mid q-1$ and $F_i = x_i^{d_i} - x_0^{d_0\Tilde{w}_i}$, $i=1,2,\ldots,n$,  then, the vanishing ideal of $Y_Q$ is the following complete intersection lattice ideal:
$$I(Y_Q) = \langle F_1, F_2, \ldots, F_n \rangle.$$
\end{proposition}
\begin{proof} Since $D=diag (d_0,\ldots,d_n)$ and $\bb=[ w_0 \cdots w_n]$, it follows that their product is $\bb D=[ w_0d_0 \cdots w_nd_n]$. It is clear that $\Tilde{w}_i(q-1)=w_id_i$, and so 
$$\gcd(w_0d_0,\dots,w_nd_n)=(q-1)\gcd(\Tilde{w}_0,\dots,\Tilde{w}_n).$$ Therefore, we have the equality of the lattices $L_{\bb D}=L_{\Tilde{W}}$, where $\Tilde{W}$ is the matrix with columns $\Tilde{w}_i$, for $i=0,\dots,n$.

When $w_0 \mid q-1$, we have $\Tilde{w}_0 =1$ and thus the lattice $L_{\Tilde{W}}$ has the following basis $$\{(-\Tilde{w}_1,\e_1),\dots,(-\Tilde{w}_n,\e_n)\},$$ where $\e_i$ form the standard basis for $\Z^n$.
 Consider the matrix $M$ whose columns are the basis vectors of $L_{\Tilde{W}}$ given above. Since the matrix $M$ is mixed-dominating, it follows from Theorem \ref{t:mixeddominating} that the lattice ideal of $L_{\Tilde{W}}$ is a complete intersection generated by the binomials $x_i - x_0^{\Tilde{w}_i}$, $i=1,2,\ldots,n$.

By Theorem \ref{t:IdealDegTori}, the vanishing ideal $I(Y_Q)$ is the binomial ideal $I_L$ for the lattice $L=D(L_{\beta D})$, whose generators are obtained substituting $x_i^{d_i}$ for $x_i$ in the binomials above generating the lattice ideal of $L_{\Tilde{W}}$, by Proposition \ref{p:IdealDegTori}. Therefore, the vanishing ideal $I(Y_Q)$ is a complete intersection generated by the binomials $F_1, F_2, \ldots, F_n.$
\end{proof}

\begin{proposition}\label{P:HilbertF} Let $Q=diag(w_0,\ldots,w_n)$ and $Y_Q=\{[{t_0}^{w_{0}}:\ldots:{t_n}^{w_{n}}]|t_i\in \F_q^{*}\}$ be the corresponding subgroup of $T_X$ for $X=\Pp(w_0,\ldots,w_n)$.
 If $w_0 \mid q-1$ then, for any $\aa\in \N \bb$ we have
$$
H_{Y_Q}(\aa)=\sum_{s=0}^{n}\,\, (-1)^s\!\!\! \sum_{I\subseteq [n], |I|=s} \dim_{\K}S_{\aa-\aa_I},
$$
where $\aa_I=\sum_{i\in I}\aa_i$. Moreover, the $a$-invariant of $Y_Q$ is given by the formula $ a_{Y_Q}=(d_1-1)w_1+\cdots+(d_n-1)w_n-w_0$.
\end{proposition}

\begin{proof} 
 Notice that $I(Y_Q)$ is a complete intersection by Proposition \ref{prop:ideal_Y_Q} generated by binomials of degrees $\aa_1=d_1w_1,\dots,\aa_n=d_nw_n$.  Thus, its minimal free resolution is given by the Koszul complex. As in the proof of \cite[Proposition 3.13]{SaSop16JA} we have the following exact sequence 
$$\dis { 0  \rightarrow W_n \rightarrow \cdots \rightarrow  W_s  \rightarrow \cdots \rightarrow  W_1 \rightarrow S_{\aa}\rightarrow (S/I(Y_Q))_{\aa}\rightarrow 0 },$$
where, for every $s=1,\dots,n$, the vector space $W_s$ is given by 
$$W_s=\bigoplus_{I\subseteq [n], |I|=s} S(-\aa_{I})_{\aa}=\bigoplus_{I\subseteq [n], |I|=s} S_{\aa-\aa_{I}}.$$
Therefore, we obtain:
 \begin{eqnarray}\label{e:HF} 
 H_{Y_Q}(\aa)&=&\dim_{\K}S_{\aa}+\sum_{s=1}^{n} (-1)^s \dim_{\K} W_s\nonumber\\
&=& \sum_{s=0}^{n} (-1)^s \sum_{I\subseteq [n], |I|=s} \dim_{\K}S_{\aa-\aa_I},
 \end{eqnarray}
where $\aa_I=\sum_{i\in I}\aa_i$.
By Proposition 8.23 in \cite{CombComAlgBook}, the numerator of the Hilbert series in Equation \ref{e:HilbSeries} is as follows: 
$$p_{S/I(Y_Q)}=\sum_{s=0}^{n} (-1)^s \sum_{I\subseteq [n], |I|=s} t^{\aa_I}.$$  Hence, $p_{S/I(Y_Q)}$ has degree $\aa_1+\cdots+\aa_n=d_1w_1+\cdots+d_nw_n$, and thus $$ a_{Y_Q}=(d_1-1)w_1+\cdots+(d_n-1)w_n-w_0$$ by Proposition \ref{p:a-invariant}. 
\end{proof}

\begin{example} Let $X=\Pp(1,1,2)$. Consider the matrix $Q=diag(1,1,2)$ and $Y_Q=\{[t_0:t_1:t_2^{2}] \:| \: t_0,t_1,t_2\in \F_q^{*}\}$. Assume that $q$ is odd. So, we have \[(d_0,d_1,d_2)=(q-1,q-1,(q-1)/2) \text{ and } (\Tilde{w}_0, \Tilde{w}_1,\Tilde{w}_2) = (1,1,1).
\]
Thus,  $I(Y_Q) =\langle F_1, F_2 \rangle= \langle x_1^{q-1} - x_0^{q-1}, x_2^{(q-1)/2} - x_0^{q-1} \rangle$. As the degrees of the generators are $\aa_1=q-1$ and $\aa_2=q-1$, a graded minimal free resolution of $I(Y_Q)$ is given by:
$$\dis { 0  \rightarrow S_{\aa-\aa_1-\aa_2} \stackrel{[-F_2 \: F_1]^T}{\longrightarrow} S_{\aa-\aa_1} \oplus S_{\aa-\aa_2} \stackrel{[F_1 \: F_2]}{\longrightarrow} S_{\aa}\rightarrow (S/I(Y_Q))_{\aa}\rightarrow 0 }.$$
Therefore, the Hilbert function is computed to be
\begin{eqnarray}\label{ex:HF} 
 H_{Y_Q}(\aa)&=&\dim_{\K}S_{\aa}-\dim_{\K}S_{\aa-\aa_1}-\dim_{\K}S_{\aa-\aa_2}+\dim_{\K}S_{\aa-\aa_1-\aa_2}\nonumber \\
 &=& \dim_{\K}S_{\aa}-2\dim_{\K}S_{\aa-(q-1)}+\dim_{\K}S_{\aa-2(q-1)}\nonumber.
 \end{eqnarray}
 We first notice the following
 $$\dim_{\K}S_{\aa}=\left\{ 
\begin{array}{ll}
 (\aa_0+1)^2    & \mbox{if } \aa=2\aa_0\\
  (\aa_0+1)(\aa_0+2)    & \mbox{if } \aa=2\aa_0+1.
\end{array}
\right.
$$
 
Thus, if $0\leq \aa \leq q-2$, then $\dim_{\K}S_{\aa-(q-1)}=\dim_{\K}S_{\aa-2(q-1)}=0$. Hence,

$$H_{Y_Q}(\aa)=\left\{ 
\begin{array}{ll}
 (\aa_0+1)^2    & \mbox{if } \aa=2\aa_0\\
  (\aa_0+1)(\aa_0+2)    & \mbox{if } \aa=2\aa_0+1.
\end{array}
\right.
$$
When, $q-1\leq \aa<2(q-1)$, we have $\dim_{\K}S_{\aa-2(q-1)}=0$. It is easy to see that
$$\dim_{\K}S_{\aa-(q-1)}=\left\{ 
\begin{array}{ll}
 (\aa_0+1-(q-1)/2)^2    & \mbox{if } \aa=2\aa_0\\
  (\aa_0+1-(q-1)/2)(\aa_0+2-(q-1)/2)    & \mbox{if } \aa=2\aa_0+1.
\end{array}
\right.
$$
Hence, we have the following formula for $H_{Y_Q}(\aa):$
$$\left\{ 
\begin{array}{ll}
 (\aa_0+1)^2 -2(\aa_0+1-(q-1)/2)^2    & \mbox{if } \aa=2\aa_0\\
  (\aa_0+1)(\aa_0+2) -2(\aa_0+1-(q-1)/2)(\aa_0+2-(q-1)/2)    & \mbox{if } \aa=2\aa_0+1.
\end{array}
\right.
$$
Finally, when $\aa \geq 2(q-1)$, we get
$$\dim_{\K}S_{\aa-2(q-1)}=\left\{ 
\begin{array}{ll}
 (\aa_0+1-(q-1))^2    & \mbox{if } \aa=2\aa_0\\
  (\aa_0+1-(q-1))(\aa_0+2-(q-1))    & \mbox{if } \aa=2\aa_0+1.
\end{array}
\right.
$$
Therefore, we have
 $H_{Y_Q}(\aa)=(q-1)^2/2=|Y_Q|$ which is not surprising as we have $\aa > a_{Y_Q}$ in this case.

\end{example}

\section{Length and Dimension when $X=\Pp(1,w_1,\dots,w_n)$}

 Let $\mathbb{F}_q^* = \langle \eta \rangle$, then the order of $\eta_i := \eta^{w_i}$ is 
$$ d_i = \frac{q-1}{\gcd(q-1,w_i)} \quad i=1,\ldots,n. $$ 
By using $I(Y_Q)$,  the length and the dimension of
$\cC_{\alpha,Y_Q}$  are computed as follows.
\begin{tm} \label{t:dimension}
Let $X=\Pp(1,w_1,\dots,w_n)$ be a weighted projective space over the field $\overline{\F}_q$. Consider $Q = \mbox{diag}(1,w_1,\ldots,w_n)$ and the subgroup it defines in  $T_X(\F_q)$:
 $$Y_Q = \{[t_0:t_1^{w_1}:\ldots :t_n^{w_n}] \: | \: t_i \in \F_q^*, \text{ for all } i=0,\ldots,n\}.$$ Then, the length of
$\cC_{\alpha,Y_Q}$ is $|Y_Q| = d_1\cdots d_n$ and the dimension is
$$\displaystyle \dim(\cC_{\alpha,Y_Q})= \sum_{m_n=0}^{\min\{\lfloor \frac{\alpha}{w_n} \rfloor,d_n-1\}}\sum_{m_{n-1}=0}^{\min\{\lfloor \frac{\alpha-m_nw_n}{w_{n-1}}\rfloor,d_{n-1}-1\}}\cdots \sum_{m_1=0}^{\min\{\lfloor \frac{\alpha-m_n w_n-\cdots -m_2w_2}{w_1}\rfloor,d_1-1\}}1.$$
Moreover, the a-invariant is given by
$$ a_{Y_Q}=(d_1-1)w_1+\cdots+(d_n-1)w_n-1.$$
\end{tm}

\begin{proof}
We first prove that 
\begin{equation} \label{e:ProCyclics}
Y_Q=\langle [1:\eta_1:1:\ldots :1] \rangle \times \cdots \times \langle [1:\ldots:1:\eta_n] \rangle.
\end{equation}
Multiplying by $[\lambda:\lambda^{w_1}:\ldots :\lambda^{w_n}]$ does not change an equivalence class for every $\lambda\in \F_q^*$. So, we have the equality of the following points:
$$[t_0:t_1^{w_1}:\ldots :t_n^{w_n}]=[1:(t_1/t_0)^{w_1}:\ldots :(t_n/t_0)^{w_n}].$$ Hence, we have
$$Y_Q=\{[1:s_1^{w_1}:\ldots :s_n^{w_n}] \: | \: s_i \in \F_q^*, \text{ for all } i=1,\ldots,n\}.$$ 
Since $s_i=\eta^{k_i}$, for some $k_i\in \N$, it is clear that $s_i^{w_i}=\eta_i^{k_i}$ and thus
$$Y_Q=\{[1:\eta_1^{i_1}:\ldots :\eta_n^{i_n}]  \: | \: 0\leq i_1\leq d_1, \dots, 0\leq i_n\leq d_n\},$$ 
from which the claim in (\ref{e:ProCyclics}) is deduced, and thus $|Y_Q|=d_1\cdots d_n.$

If $w_0=1$, then $d_0=q-1$ and so the vanishing ideal of $Y_Q$ is generated by the binomials $F_i = x_i^{d_i} - x_0^{d_i {w}_i}$, for $i=1,2,\ldots,n$. With respect to any term order  for which $x_0$ is the smallest variable, the leading monomial of $F_i$ is clearly $x_i^{d_i}$. Since the monomials $x_i^{d_i}$ and $x_j^{d_j}$ are relatively prime for different $i$ and $j$, it readily follows that the binomials $F_1,\dots, F_n$ form a Groebner basis for the vanishing ideal $I(Y_Q)$. It is well-known ( \cite[p.232]{CLO2007}) then that a basis for the vector space $S_{\aa}/I_{\aa}(Y_Q)$ is given by the monomials  $\x^{\m}=x_0^{m_0}x_1^{m_1}\cdots x_n^{m_n}$ of degree $\aa$ that can not be divided by the leading monomials $x_i^{d_i}$ of $F_i$, for all $i=1,2,\ldots,n$ and for 
\[\aa=m_0+m_1w_1+\cdots+m_nw_n\in \N=\langle 1,w_1,\dots,w_n \rangle.
\]
Therefore, a basis for $S_{\aa}/I_{\aa}(Y_Q)$ corresponds to the set of tuples $(m_0,m_1,\dots,m_n)$ satisfying $\aa=m_0+m_1w_1+\cdots+m_nw_n$ and $m_i \leq d_i-1$, for all $i=1,2,\ldots,n$. The elements of this set can be identified step by step as we explain now. We start first by choosing an integer $m_n$ between $0$ and $\min \{\lfloor \frac{\alpha}{w_n} \rfloor, d_n-1\}$ and observe that the elements of the set in question can be partitioned into subsets for every choice of $m_n$ in the aforementioned range. More precisely, for each fixed $m_n$, we have a subset consisting of tuples $(m_0,m_1,\dots,m_n)$ satisfying 
\[m_0+m_1w_1+\cdots+m_{n-1}w_{n-1}=\aa-m_nw_n \text{ and } m_i \leq d_i-1, \text{ for all } i=1,2,\ldots,n-1.\]
As a second step, we fix $m_{n-1}$ between $0$ and $\min \{\lfloor \frac{\aa-m_nw_n}{w_{n-1}} \rfloor, d_{n-1}-1\}$, and look for the solutions $(m_0,m_1,\dots,m_{n-2})$ satisfying 
\[m_0+m_1w_1+\cdots+m_{n-2}w_{n-2}=\aa-m_nw_n-m_{n-1}w_{n-1} \text{ and } m_i \leq d_i-1,\]
for all $i=1,2,\ldots,n-2$.
Continuing inductively, we end up with a unique $m_0$ satisfying 
\[m_0=\aa-m_nw_n-m_{n-1}w_{n-1}-\cdots-m_1w_1.
\]
Hence, the dimension of the code, which is nothing but the dimension of the vector space $S_{\aa}/I_{\aa}(Y_Q)$, is exactly the sum given by the formula
\[ \displaystyle \dim(\cC_{\alpha,Y_Q})= \sum_{m_n=0}^{\min\{\lfloor \frac{\alpha}{w_n} \rfloor,d_n-1\}}\sum_{m_{n-1}=0}^{\min\{\lfloor \frac{\alpha-m_nw_n}{w_{n-1}}\rfloor,d_{n-1}-1\}}\cdots \sum_{m_1=0}^{\min\{\lfloor \frac{\alpha-m_n w_n-\cdots -m_2w_2}{w_1}\rfloor,d_1-1\}}1.
\]
 The $a-$invariant can be obtained from Proposition \ref{P:HilbertF}, by substiting $w_0=1$.
\end{proof}

\section{Codes on $Y_Q \subset \Pp(1,1,a)$}
For any positive integer $a$, we compute the basic parameters of the code 
$\cC_{\alpha,Y_Q}$, for the subgroup $Y_Q = \{[t_0:t_1:t_2^a] \: | \: t_0,t_1,t_2 \in \F_q^*\}$ of $T_X(\F_q)$ for the weighted projective space $X=\Pp(1,1,a)$.
\begin{tm} \label{t:codesOnDegenToriP(1,1,a)} Let $d_2= \frac{q-1}{\gcd(a,q-1)}$, $k=\lfloor \frac{\alpha-(q-2)}{a} \rfloor$ and $\mu_2=\min \{\lfloor \frac{\alpha}{a} \rfloor, d_2-1\}$.
Then, the length of $\cC_{\alpha,Y_Q}$ is $N=|Y_Q| = (q-1)d_2$. Its dimension $K(\cC_{\alpha,Y_Q})$ is 
\[ \displaystyle
\begin{array}{ll}
(\mu_2+1)(\aa+1-\mu_2a/2), & \mbox{if } 0\leq \alpha \leq q-2 \\
(q-1)(k+1)+(\mu_2-k)[\aa+1-(\mu_2+k+1)a/2], & \mbox{if } 0 < \aa- (q-2)<(d_2-1)a \\
N   & \text{otherwise}.
\end{array}
\]
and the minimum distance of $\cC_{\alpha,Y_Q}$ is:
$$\delta(\cC_{\alpha,Y_Q})=\left\{ 
\begin{array}{ll}
  d_2(q-1-\alpha)   & \mbox{if } 0\leq \alpha \leq q-2 \\
    d_2-k & \mbox{if } q-2 \leq \aa <(q-2)+(d_2-1)a \\
1    & \text{otherwise}.
\end{array}
\right.
$$

\end{tm}
\begin{proof} Since $w_1=1$, we have $d_1=q-1$. It follows from Equation \ref{e:ProCyclics} that $$Y_Q=\{[1:\eta_1^{i_1}:\eta_2^{i_2}] \: | \: 0\leq i_1 \leq d_1 \text{ and } 0\leq i_2 \leq d_2\},$$
so the length of the code is $d_1d_2=(q-1)d_2$.

When $0\leq \alpha \leq q-2$, the dimension formula in Theorem \ref{t:dimension} specializes to 
\begin{eqnarray}
\displaystyle \dim(\cC_{\alpha,Y_Q}) 
\nonumber &=& \sum_{m_2=0}^{\mu_2}\sum_{m_{1}=0}^{\min\{\alpha-m_2 a,q-2\}}1 = \sum_{m_2=0}^{\mu_2}\sum_{m_{1}=0}^{\alpha-m_2 a}1 \\
\nonumber &=&\sum_{m_2=0}^{\mu_2} (\alpha-m_2 a+1)=(\mu_2+1)(\aa + 1)-a\sum_{m_2=0}^{\mu_2} m_2\\
\nonumber &=&(\mu_2+1)(\aa + 1)-a\frac{\mu_2(\mu_2 +1)}{2}.
\end{eqnarray}

If $q-2 < \aa <(q-2)+(d_2-1)a$, then using the formula in Theorem \ref{t:dimension} again, we get
\begin{eqnarray}
\displaystyle \dim(\cC_{\alpha,Y_Q}) 
\nonumber &=& \sum_{m_2=0}^{\mu_2}\sum_{m_{1}=0}^{\min\{\alpha-m_2 a,q-2\}}1 \\
\nonumber &=& \sum_{m_2=0}^{k}\sum_{m_{1}=0}^{q-2}1 + \sum_{m_2=k+1}^{\mu_2}\sum_{m_{1}=0}^{\alpha-m_2 a}1 \\
\nonumber &=&(q-1)(k+1)+\sum_{m_2=k+1}^{\mu_2} (\alpha-m_2 a+1)\\
\nonumber &=&(q-1)(k+1)+(\mu_2-k)(\aa + 1)-a\sum_{m_2=k+1}^{\mu_2} m_2 \\
\nonumber &=&(q-1)(k+1)+(\mu_2-k)(\aa + 1)-a\frac{\mu_2(\mu_2 +1)-k(k+1)}{2}.
\end{eqnarray}
Notice that these dimensions are the number of lattice points of the polygons depicted below.
\begin{figure}[htb!]
\centering
\begin{minipage}{0.5\textwidth}
  \centering
\begin{tikzpicture}[scale=0.5]
\def\q{7};
\def\cOne{\q};
\def\cTwo{\q-2};
\def\l{2};

\def\bPrime{3};
\pgfmathsetmacro\c{\q-2};
\pgfmathsetmacro\cl{0.5*(\c)};
			
			\pgfmathsetmacro\qMinusOne{\q-1};
			\pgfmathsetmacro\d{2};
            \pgfmathsetmacro\b{2};

\draw[step=1cm,gray!30!white,very thin] (0,0) grid (\q-1,\q-1);
\draw[thick,->] (0,0) -- (\q-1,0) node[anchor=south east] {};
\draw[thick,->] (0,0) -- (0,\q-1) node[anchor=south east] {};
\draw (\c,0) -- (0,\cl);

\node [below] at (\cOne-1,0) {$q$-$2$};

\node [above right] at (\c-\l*\b,\b) {$m_1+ m_2a=\alpha$};
\node [left] at (-0.12,\cl+0.14) {$\alpha$/a};

\node [left] at (0,\b) {$\mu_2$};
\node [below] at (\c,-0.12) {$\alpha$};
\node [below left] at (0,0) {$0$};
\node [below] at (\c-\l*\b,0) {$\alpha$- $\mu_2 a$};
\filldraw[fill=blue!40!white, draw=black] (0,0) rectangle (\c-\l*\b,\b);
\filldraw[fill=red!40!white, draw=black] (\c-\l*\b,0)-- (\c-\l*\b,\b)--(\c,0)--cycle;
\end{tikzpicture} 
  \caption{$\alpha\leq q-2$}
  \label{fig:case1a}
\end{minipage}%
\begin{minipage}{0.5\textwidth}
  \centering
    \begin{tikzpicture}[scale=0.5]
\def\q{11};
\def\cOne{\q-5};
\def\cTwo{7};
\def\l{2};

\def\bPrime{2};
\pgfmathsetmacro\c{\q-1};
\pgfmathsetmacro\cl{0.5*(\c)};
			
			\pgfmathsetmacro\qMinusOne{\q-1};
			\pgfmathsetmacro\d{3};
            \pgfmathsetmacro\b{3};

\draw[step=1cm,gray!30!white,very thin] (0,0) grid (\q-1,\q-5);
\draw[thick,->] (0,0) -- (\q-1,0) node[anchor=south east] {};
\draw[thick,->] (0,0) -- (0,\q-5) node[anchor=south east] {};
\draw (\c,0) -- (0,\cl);

\draw [dashed, red, thick] (\cOne-1,0) -- (\cOne-1,0.5+\bPrime) --(0,0.5+\bPrime);
\node [below] at (\cOne-1,0) {$q$-$2$};

\node [above right] at (\c-\l*\bPrime-2,\bPrime+1) {$m_1+m_2a=\alpha$};
\node [left] at (0,\cl) {$\alpha/a$};
\node [left] at (0,\bPrime+2) {$\mu_2$};
\node [left] at (0,\bPrime+1) {$k+1$};
\node [left] at (0,\bPrime) {$k$};
\node [below] at (\c,0) {$\alpha$};
\node [below left] at (0,0) {$0$};
\node [below] at (\c-\l*\bPrime,0) {};
\filldraw[fill=blue!40!white, draw=black] (0,0) rectangle (\cOne-1,\bPrime);
\filldraw[fill=red!40!white, draw=black] (0,1+\bPrime)--(\c-\l*\bPrime-\l,\bPrime+1)--(\c-\l*\bPrime-\l-2,\bPrime+2)-- (0,\bPrime+2)--cycle;
\end{tikzpicture}
  \caption{$\alpha>q-2$}
  \label{fig:case1b}
\end{minipage}
\end{figure}

As for the minimum distance, we first give an upper bound on the number $|V_{Y_Q}(F)|$ of zeroes on $Y_Q$ of a homogeneous polynomial $F$ of degree $\aa$ and  then demonstrate a specific polynomial attaining that bound.
Let $[d_2]$ denote the set of non-negative integers smaller than $d_2$, and set $$J_F:=\{j\in [d_2] \: | \: x_2-\eta_2^{j}x_0^{w_2} \text{ divides } F\}.$$
We claim that
\begin{equation}\label{e:upperbound}
 |V_{Y_Q}(F)|\leq d_1|J_F|+(d_2-|J_F|)\deg_{x_1}(F),  
\end{equation}
where $\deg_{x_1}(F)$ is the usual degree of $F$ in the variable $x_1$. The polynomial $f_j(x_1):=F(1,x_1,\eta_2^{j})\in \F_q[x_1]$ vanishes at the points $[1:\eta_1^{i}:\eta_2^{j}]$, for every $i\in [d_1]$, when $j\in J_F$. Thus, there are $d_1|J_F|$ such roots of $F$. On the other hand, $f_j$ is not a zero polynomial when $j\notin J_F$, and in this case it can have at most its degree many zeroes, giving rise to $(d_2-|J_F|)\deg_{x_1}(F)$ many roots of $F$, completing the proof of the claim.

Since we always have 
$$\dis F=\prod_{j=1}^{|J_F|} (x_2-\eta_2^{j}x_0^{w_2}) F'$$
it follows that $\deg_{x_1}(F)=\deg_{x_1}(F')\leq \aa-|J_F|w_2$. Thus, we have
\begin{eqnarray}
\nonumber |V_{Y_Q}(F)| &\leq& d_1|J_F|+(d_2-|J_F|)(\aa-|J_F|w_2)\\
 &\leq& d_2\aa+|J_F|(d_1-\aa-w_2(d_2-|J_F|)).
\end{eqnarray}
Notice that the number in the parenthesis above is
\begin{eqnarray*}
d_1-\aa-w_2(d_2-|J_F|)=d_1-\aa-w_2d_2+w_2|J_F|=d_1-(q-1)\Tilde{w}_2-\aa +w_2|J_F|
\end{eqnarray*}
which is non-positive since $d_1 \leq q-1 \leq (q-1)\Tilde{w}_2$ and $|J_F|w_2\leq \deg(F)= \aa$. Hence,
altogether, we have the upper bound
\begin{eqnarray}
|V_{Y_Q}(F)| \leq d_2\aa.
\end{eqnarray}
Consider now the following polynomial:
$$\dis F_0=\prod_{i=1}^{\aa} (x_1-\eta_1^{i}x_0^{})$$
which  vanishes at the points $[1:\eta_1^{i}:\eta_2^{j}]$, for every $i\in [\aa]$ and $j\in [d_2]$,  implying that  $|V_{Y_Q}(F_0)| = d_2\aa.$
As the weight of the codeword $\ev_{Y_Q}(F_0)$ is clearly $$|Y_Q|-|V_{Y_Q}(F_0)|=d_2(q-1)-d_2\aa$$ and that of a general codeword $\ev_{Y_Q}(F)$ is $$|Y_Q|-|V_{Y_Q}(F)| \geq d_2(q-1)-d_2\aa,$$ it follows that the minimum distance of the code is $d_2(q-1-\aa)$, when $\aa < q-1$.

When $\aa\geq a_Y+1=(q-2)+(d_2-1)a$, the code is trivial, so $\delta(\cC_{\alpha,Y_Q})=1$. 

From now on, assume that $q-2 \leq \aa < a_Y+1=(q-2)+(d_2-1)a$. Let $k$ be the quotient and $r_0$ be the remainder of the division of $\alpha-(q-2)$ by $w_2=a$, i.e.
$$\alpha-(q-2)=ka+r_0 \text{ where } 0\leq k:= \left\lfloor \frac{\alpha-(q-2)}{a} \right\rfloor \leq d_2-2 \text{ and } 0\leq r_0 \leq a-1.$$
 When $|J_F|=d_2$, $F$ vanishes on $Y_Q$, so $F$ gives a codeword with zero weight. Thus, we suppose $|J_F| \leq d_2-1$. 
 
 If $|J_F| \leq k$ also, then by $\deg_{x_1}(F) \leq d_1-1=q-2$ and (\ref{e:upperbound}) we have
\begin{eqnarray*}
 |V_{Y_Q}(F)| &\leq& (q-1)|J_F|+(d_2-|J_F|)(q-2)=(q-2)d_2+|J_F|\\
 &\leq& (q-2)d_2+k.
\end{eqnarray*}
 
If $|J_F| > k$, we let $|J_F|=k+j_0$ with $j_0\geq 1$. As $\aa=q-2+ka+r$ with $r\leq a-1$, we have $\aa-|J_F|a=\aa-(k+1)a-a(|J_F|-k-1)\leq q-2-1-a(j_0-1)$. As $\deg_{x_1}(F) \leq \aa-|J_F|a$, it follows from (\ref{e:upperbound}) that we have,
\begin{eqnarray*}
  |V_{Y_Q}(F)| &\leq& (q-1)|J_F|+(d_2-|J_F|)(\aa-|J_F|a)\\
 &\leq& (q-1)|J_F|+(d_2-|J_F|)(q-2-1-a(j_0-1))\\
  &\leq& (q-2)d_2+|J_F|+(d_2-|J_F|)(-1-a(j_0-1))
\end{eqnarray*}
Since $d_2-|J_F|\geq 1$ and $a\geq 1$, we have
\begin{eqnarray*}
  |V_{Y_Q}(F)| &\leq& (q-2)d_2+|J_F|+(d_2-|J_F|)(-1-a(j_0-1))\\
 &\leq&  (q-2)d_2+|J_F|-1-(j_0-1)=(q-2)d_2+k.
\end{eqnarray*}

We consider the following homogeneous polynomial of degree $\aa$  now:
$$\dis G_0=x_0^{r_0}\prod_{i=1}^{q-2} (x_1-\eta_1^{i}x_0^{})\prod_{j=1}^{k} (x_2-\eta_2^{j}x_0^{a})$$
which vanish at the points $[1:\eta_1^{i}:\eta_2^{j}]$, for every $i\in [q-2]$ and $j\in [d_2]$, together with the points $[1:\eta_1^{i}:\eta_2^{j}]$, for $i=q-1$ and $j\in [k]$. Therefore,  the number of roots is $|V_{Y_Q}(G_0)| =(q-2)d_2+k.$ It readily follows that the minimum distance $\delta(\cC_{\alpha,Y_Q})$ of the code is the weight $(q-1)d_2-(q-2)d_2-k=d_2-k$ of the codeword corresponding to $G_0$.
\end{proof}
 \begin{remark}
It is very difficult to give a closed formula for some parameters of the code $\cC_{\alpha,Y_Q}$, for the subgroup $Y_Q = \{[t_0:t_1^{w_1}:t_2^{w_2}] \: | \: t_0,t_1,t_2 \in \F_q^*\}$ of $T_X(\F_q)$ in the more general case of the weighted projective plane  $X=\Pp(1,w_1,w_2)$. This is mainly because of the formulas involving a division by the integer $w_1$. In this case one has to use the floor function when the ratio is not integer, which we explain in more details below. There are $4$ cases to consider:
\begin{eqnarray*}
Case ~1:  \left\lfloor\frac{\alpha}{w_1}\right\rfloor \leq d_1-1 \text{ and } \left\lfloor\frac{\alpha}{w_2}\right\rfloor \leq  d_2-1\\
Case ~2:  \left\lfloor\frac{\alpha}{w_1}\right\rfloor > d_1-1 \text{ and } \left\lfloor\frac{\alpha}{w_2}\right\rfloor \leq  d_2-1\\
Case ~3:  \left\lfloor\frac{\alpha}{w_1}\right\rfloor \leq d_1-1 \text{ and } \left\lfloor\frac{\alpha}{w_2}\right\rfloor >  d_2-1\\
Case ~4:  \left\lfloor\frac{\alpha}{w_1}\right\rfloor > d_1-1 \text{ and } \left\lfloor\frac{\alpha}{w_2}\right\rfloor >  d_2-1.
\end{eqnarray*} 
For instance, the dimension formula in Theorem \ref{t:dimension} specializes to 
\begin{eqnarray}
\displaystyle \dim(\cC_{\alpha,Y_Q}) 
\nonumber &=& \sum_{m_2=0}^{\mu_2 } \sum_{m_{1}=0}^{\min\left \{ \left\lfloor\frac{\alpha-m_2 w_2}{w_1}\right\rfloor,d_1-1\right\}} 1,
\end{eqnarray}  
where $\mu_2=\min \left\{ \left\lfloor\frac{\alpha}{w_2}\right\rfloor, d_2-1\right\}$. 

We first consider Cases 2 and 4. In these cases, as we have $\aa \geq w_1(d_1-1)$, we choose \[\mu'_2:=\left \lfloor \frac{\alpha-w_1(d_1-1)}{w_2}  \right\rfloor\] 
so that  $\min\left \{ \left\lfloor\frac{\alpha-m_2 w_2}{w_1}\right\rfloor,d_1-1\right\}=d_1-1$ for all $m_2\leq \mu'_2$. Hence, we have 
\begin{eqnarray*}
\displaystyle \dim(\cC_{\alpha,Y_Q}) 
\nonumber &=& \sum_{m_2=0}^{\mu'_2 } \sum_{m_{1}=0}^{d_1-1} 1+
 \sum_{m_2=\mu'_2+1 }^{\mu_2 } \sum_{m_{1}=0}^{\left\lfloor\frac{\alpha-m_2 w_2}{w_1}\right\rfloor} 1\\
 &=& (\mu'_2+1)  d_1 +
 \sum_{m_2=\mu'_2+1}^{\mu_2 } \left(\left\lfloor\frac{\alpha-m_2 w_2}{w_1}\right\rfloor+1\right)\\
 &=& \mu'_2  (d_1-1)+\mu_2+d_1 +
  \left\lfloor\frac{\alpha-(\mu'_2+1) w_2}{w_1}\right\rfloor+\cdots+\left\lfloor\frac{\alpha-\mu_2 w_2}{w_1}\right\rfloor.
\end{eqnarray*} 
A similar formula for $\dim(\cC_{\alpha,Y_Q})$ can be obtained in Cases $1$ and $3$. In any case, we conclude that a closed formula for the dimension is difficult to get.

 As for the minimum distance, one can generalize Theorem \ref{t:codesOnDegenToriP(1,1,a)} as follows.
Let $U(x,y)$ be a polynomial defined as
\[U(x,y):=d_1y+(d_2-y)x \text{ for } 0\leq x \leq \min \left\{ \left\lfloor\frac{\alpha- yw_2}{w_1}\right\rfloor, d_1-1\right\}, 0\leq y \leq \mu_2.
\] 
Then, the upper bound in (\ref{e:upperbound}) becomes
\begin{equation*}
 |V_{Y_Q}(F)|\leq U(\deg_{x_1}(F),|J_F|).  
\end{equation*}
Since we have
\[
 x\leq \min \left\{ \left\lfloor\frac{\alpha- yw_2}{w_1}\right\rfloor, d_1-1\right\}=\left\{ 
\begin{array}{ll}
      d_1-1 & \mbox{if } 0\leq y \leq \mu'_2\\
      \left\lfloor\frac{\alpha- yw_2}{w_1}\right\rfloor   & \mbox{if } \mu'_2< y \leq \mu_2,
\end{array}
\right.
\]
it follows that 
\[U(x,y)\leq u(y):=d_1y+(d_2-y)\min \left\{ \left\lfloor\frac{\alpha- yw_2}{w_1}\right\rfloor, d_1-1\right\}.
\] 
Therefore, we get
\[U(x,y)\leq d_1y+(d_2-y)(d_1-1)=d_2(d_1-1)+y \text{ for all } y\in [0,\mu'_2]
\] 
\[
\text{and } U(x,y)< d_1y+(d_2-y)(d_1-1) \text{ for all } y\in [\mu'_2+1,\mu_2].
\]
Clearly, the polynomial $U(x,y) $ attains the maximum value at $(d_1-1,\mu_2'$), which is $$u(\mu'_2)=d_2(d_1-1)+\mu'_2=d_1d_2-(d_2-\mu'_2).$$ Thus, the minimum distance of the code $\cC_{\alpha,Y_Q}$ will be $d_2-\mu'_2$. This means that the proof of the second part of the Theorem \ref{t:codesOnDegenToriP(1,1,a)} can be generalized very easily via replacing $q-1$ (resp. $q-2$) by $d_1$ (resp. $d_1-1$) to the Cases $2$ and $4$, namely for the values of $\aa$ satisfying $w_1(d_1-1)\leq \aa <w_1(d_1-1)+w_2(d_2-1)$.

 When $w_1=1$, as in the proof of the first part of the Theorem \ref{t:codesOnDegenToriP(1,1,a)}, the ratio $\frac{\alpha- yw_2}{w_1}$ was an integer yielding $x=\left\lfloor\frac{\alpha- yw_2}{w_1}\right\rfloor=\alpha- yw_2$, and so the upper bound was
 \[u(y)=d_1y+(d_2-y)(\alpha- yw_2)=w_2y^2+(d_1-\aa-d_2w_2)y+d_2\alpha.
 \]
 As the quadratic polynomial $u(y)$  is concave up, it was clear that the absolute maximum is attained at the boundary points of the interval $[0,\lfloor \aa/a\rfloor]$ and we were able to prove that $u(0)$ was the maximum value of $U$. However, the proof of the first case does not generalize to the Cases $1$ and $3$ as the maximum values are sometimes attained at interior points. 
 
 For instance, consider the case $q=31$, $w_1=8$, $w_2=9$ and $\aa=34$. Then, we have $d_1 = 15$, $d_2 = 10, \left\lfloor\frac{\alpha}{w_1}\right\rfloor=4$ and $\left\lfloor\frac{\alpha}{w_2}\right\rfloor=3$. The function
 \[U(x,y):=15y+(10-y)x \text{ for } 0\leq x \leq \left\lfloor\frac{34- 9y}{8}\right\rfloor, 0\leq y \leq 3
\] 
has the upper bound given by  
 \[u(y)=15y+(10-y)\left\lfloor\frac{34- 9y}{8}\right\rfloor, 0\leq y \leq 3.
 \]
Notice that $[u(0),u(1),u(2),u(3)]=[40, 42, 46, 45]$. Therefore, the maximum value $46$ is attained at the interior point $y=2$. 

As the value $u(0)=\left\lfloor\frac{\alpha}{w_1}\right\rfloor$ gives rise to an upper bound on the minimum distance in Cases $1$ and $3$, we have the following:
  
 \begin{tm} \label{t:codesOnDegenToriP(1,a,b)} Let $d_1= \frac{q-1}{\gcd(w_1,q-1)}$, $d_2= \frac{q-1}{\gcd(w_2,q-1)}$ and $k=\left\lfloor \frac{\alpha-w_1(d_1-1)}{w_2} \right\rfloor$.
Then, the length of $\cC_{\alpha,Y_Q}$ is $N=|Y_Q| = d_1d_2$. The minimum distance of $\cC_{\alpha,Y_Q}$ satisfies
$$
\begin{array}{ll}
  \delta(\cC_{\alpha,Y_Q}) \leq d_2(d_1- \left\lfloor\frac{\alpha}{w_1}\right\rfloor)   & \mbox{if } 0\leq \alpha \leq w_1(d_1-1) \\
    \delta(\cC_{\alpha,Y_Q})=d_2-k & \mbox{if } w_1(d_1-1)\leq \aa <w_1(d_1-1)+w_2(d_2-1) \\
\delta(\cC_{\alpha,Y_Q})=1    & \text{otherwise}.
\end{array}
$$

\end{tm}

\end{remark}

We conclude the paper by showcasing an example with codes having good parameters obtained by our construction.
\begin{example}
Take $a=2$, $q=5$. So, $d_2=2$ and length is $d_2(q-1) = 8$. Table \ref{tab:my_table1} exhibits the main parameters of the code $\cC_{\alpha,Y_Q}$ for $\aa$ in the first column. According to Markus Grassl's Code Tables \cite{Grassl:codetables} a best-possible code with $N=8$ has $K+\delta=8$ or $K+\delta=9$ (MDS codes). This example provides us with $3$ best possible codes whose parameters satisfy $K+\delta=8$ together with an MDS code $[8,7,2]$.
\begin{table}[ht]
    \centering
    \caption{a=2 and q=5 }
    \begin{tabular}{|c|c|}
    \hline
	{$\mathbf{\alpha}$ \cellcolor{gray!20}}& 
	{$[N, K, \delta]$ \cellcolor{gray!20}}\\

    \hline
         $0$ & $[8,1,8]$ \\
         \hline
       $1$ & $[8,2,6]$ \\
         \hline
        $2$ & $[8,4,4]$ \\
         \hline
         $3$ & $[8,6,2]$ \\
         \hline
        $4$ & $[8,7,2]$ \\
         \hline
         $5$ & $[8,1,8]$ \\
         \hline
    \end{tabular}
    \label{tab:my_table1}
\end{table}

\end{example}

 \begin{example}
 Similarly, we take $a=3$ and $q=5$ so that $d_2=4$ and length is $d_2(q-1)=4\cdot4=16$. Table \ref{tab:my_table2} gives the main parameters of the corresponding codes. 

 \begin{table}[ht]
    \centering
    \caption{a=3 and q=5 }
    \begin{tabular}{|c|c|}
    \hline
	{$\mathbf{\alpha}$ \cellcolor{gray!20}}& 
	{$[N, K, \delta]$ \cellcolor{gray!20}}\\

    \hline
         $0$ & $[16,1,16]$ \\
         \hline
       $1$ & $[16,2,12]$ \\
         \hline
        $2$ & $[16,3,8]$ \\
         \hline
         $3$ & $[16,5,4]$ \\
         \hline
        $4$ & $[16,6,4]$ \\
         \hline
         $5$ & $[16,7,4]$ \\
         \hline
         $6$ & $[16,9,3]$ \\
         \hline
       $7$ & $[16,10,3]$ \\
         \hline
        $8$ & $[16,11,3]$ \\
         \hline
         $9$ & $[16,13,2]$ \\
         \hline
        $10$ & $[16,14,2]$ \\
         \hline
         $11$ & $[16,15,2]$ \\
         \hline
    \end{tabular}
    \label{tab:my_table2}
\end{table}
\end{example}

\section{acknowledgement}
The authors would like to thank two anonymous reviewers for their careful reading and helpful suggestions which improved the presentation of the paper.

\end{document}